\documentclass[12pt]{amsart}
\usepackage{amsmath,amssymb,amsbsy,amsfonts,amsthm,latexsym,
             amsopn,amstext,amsxtra,euscript,amscd}

\def\FF{\mathbb{F}}

\DeclareMathAlphabet{\curly}{U}{rsfs}{m}{n}

\newtheorem{thm}{Theorem}
\newtheorem{cor}[thm]{Corollary}
\newtheorem{lem}[thm]{Lemma}

\theoremstyle{remark}
\newtheorem{remark}{Remark}

\def\cA{{\mathcal A}}
\def\cB{{\mathcal B}}

\def\cF{{\mathcal F}}
\def\cG{{\mathcal G}}
\def\cH{{\mathcal H}}

\def\cK{{\mathcal K}}
\def\cL{{\mathcal L}}

\def\cP{{\mathcal P}}
\def\cQ{{\mathcal Q}}
\def\cR{{\mathcal R}}

\def\cU{{\mathcal U}}

\def\cX{{\mathcal X}}

\def\cZ{{\mathcal Z}}

\def\fB{{\mathfrak B}}
\def\fL{{\mathfrak L}}

\newcommand{\ZZ}{{\mathbb Z}}
\newcommand{\QQ}{{\mathbb Q}}
\newcommand{\RR}{{\mathbb R}}

\newcommand{\NN}{{\mathbb N}}
\newcommand{\Fp}{{\mathbb F_p}}

\newcommand{\ab}{{\mathbf a}}
\newcommand{\bb}{{\mathbf b}}

\newcommand{\tu}{{\tilde u}}

\newcommand{\tPh}{{\tilde\Phi}}

\renewcommand{\vec}[1]{\mathbf{#1}}

\newcommand{\eps}{\ensuremath{\varepsilon}}

\newcommand{\fl}[1]{{\ensuremath{\left\lfloor {#1} \right\rfloor}}}

\newcommand{\conv}{\text{conv}\,}

\newcommand{\be}{\begin{equation}}
\newcommand{\ee}{\end{equation}}
\newcommand{\benn}{\begin{equation*}}   
\newcommand{\eenn}{\end{equation*}}

\renewcommand{\AA}{\curly A}
\renewcommand{\(}{\left(}
\renewcommand{\)}{\right)}
\def\fl#1{\left\lfloor#1\right\rfloor}
\def\rf#1{\left\lceil#1\right\rceil}

\def\mand{\qquad\mbox{and}\qquad}
\def\({\left(}
\def\){\right)}

\numberwithin{equation}{section}

\title[Elements of  Cosets of Small Subgroups]
{Distribution of Elements of  Cosets of Small Subgroups
and Applications}

\author[J. Bourgain]{Jean Bourgain}\address{School of Mathematics,
Institute for Advanced Study,
Princeton, NJ 08540, USA} \email{bourgain@math.ias.edu}

\author[S.~V. Konyagin]{Sergei Konyagin}\address{Steklov Mathematical
Institute,
8, Gubkin Street, Moscow, 119991, Russia}
\email{konyagin@mi.ras.ru}

\author[I.~E. Shparlinski]{Igor Shparlinski}
\address{Department of
Computing,
Macquarie University, North Ryde,
Sydney, NSW 2109, Australia}
\email{igor.shparlinski@mq.edu.au}

\date{\today}

\begin{document}

\begin{abstract}
We obtain a series of estimates on the number of small integers
and small order Farey fractions which belong to a given
coset of a subgroup of order $t$ of the group of
units of the residue ring modulo a prime $p$, in the case when $t$ is
small compared to $p$. We give two applications of these results:
to the simultaneous distribution of two high degree monomials
$x^{k_1}$ and $x^{k_2}$ modulo $p$ and to a question of
J.~Holden and P.~Moree on fixed points of the discrete logarithm.
\end{abstract}

\maketitle

%
\section{Introduction}
%

\subsection{Estimates for the number of elements of small height
in a coset of a small subgroup}

We fix a prime number $p>2$. By $\Fp$ we denote the field of residues modulo
$p$. For any element $x\in\Fp$ we define its integer height
$$|x|=\min\{|a|~:~a\in\ZZ, \ a\equiv x \pmod p \}$$
and its rational height
$$\|x\|=\min\{\max(|a|,b)~:~a\in\ZZ,b\in\NN,a\equiv bx \pmod p\}.$$
Note that by pigeonhole principle,
$$|x|\le p/2,\quad\|x\|\le\sqrt p.$$
Moreover, if $\|x\|\le\sqrt{p/2}$ then the numbers $a\in\ZZ$, $b\in\NN$
with $|a|\le\|x\|,b\le\|x\|$, $a\equiv bx \pmod p$ are uniquely defined.
Also, the rational height is defined for a rational number $x$ as
$$\|x\|=\min\{\max(|a|,b)~:~a\in\ZZ,b\in\NN,x=a/b\}.$$

As usual, we use $\FF_p^*=\FF_p\setminus\{0\}$ to denote the multiplicative group 
of $\FF_p$. 

If $t\mid p-1$ then there is a unique multiplicative subgroup
$G\subseteq\FF_p^*$. For $a\in\FF_p$ we denote
$$aG=\{ag~:~g\in G\}.$$
Our aim is to estimate the cardinality of the sets
$$U(k,t,a)=\{x~:~x\in aG, |x|\le k\},$$
$$V(k,t,a)=\{x~:~x\in aG, \|x\|\le k\}.$$

These quantities are important for estimates of exponential sums
in $\Fp$ and for analysis of distribution of cosets of $G$ in
$\Fp$. Estimates for $\#U(k,t,a)$ and for $\#V(k,t,a)$
have been obtained in~\cite{KS} and~\cite{BKS} where the case
of ``large" $t$ and $k$, that is, for $\log k\asymp\log t\asymp\log p$
(where $A\asymp B$
means that $A = O(B)$ and $B = O(A)$) has been  studied.
In this paper our main interest is related to the case of small $G$
($\log t=o(\log p)$; in particular, $\log t\asymp\log\log p$).
It is proved in~\cite{Bo2} that in a very general situation
with rather small $t$ and rather large $k$ we have
$\#U(k,t,a)=o(t)$. For smaller $k$ (say, $k\le p^{0.1}$)
the problem is easier. In this paper we prove some explicit estimates
of $\#U(k,t,a)$ and for $\#V(k,t,a)$ for such $k$ and small $G$
and apply these results to the problem of simultaneous distribution of
two powers in $\Fp$ (see~\cite{BBK}). If both parameters $t$ and $k$
are very small we establish some upper bounds for $\#U(k,t,a)$ and for
$\#V(k,t,a)$, usually much better than  the trivial
estimates
$$
\#U(k,t,a) \le \min\{2k, t\} \mand
\#V(k,t,a)  \le \min\{2k^2, t\}.
$$
These results are applied
to estimation of fixed points of the discrete logarithms.

We also noted that in the case when $k$ and $t$ are of about the
same size one can estimate $\#U(k,t,a)$ by using the results 
and techniques of~\cite{ChaShp}, based on~\cite[Theorem~1.1]{BouGar} 
that gives an explicit version of the 
sum-product theorem and of~\cite{CilGar} which is based on 
estimates of~\cite{CilJimUr} on the number of divisors in a 
short interval of an integer $n$


To formulate our results, we need some notation.

Let $x, y>0$. A positive integer $n$ is called $y$-smooth
if it is composed of prime numbers  up to $y$. The $\Psi(x,y)$
function is defined as the number of $y$-smooth positive integers
that are up to $x$.

As usual, we use $(a,b)$ to denote the greatest common divisor of
integers $a$ and $b$ (with $a^2 + b^2 > 0$).

Finally, we also use  $p_k$  to denote the $k$th prime.

We fix a prime number $p>2$. For any positive integer
$1<k<p/2$ we define  the quantity
\begin{equation}
\label{eq:r0}
r_0(k)=\fl{\frac{\log(p/2)}{\log k}}
\end{equation}
Furthermore,  let $t\in\NN$ be another parameter, then we define
\begin{equation}
\label{eq:s0}
s_0(k,t)=\max\left\{s~:~\binom{r_0(k)+s}{s}\le t\right\}.
\end{equation}

We observe that $r_0(k)$ decreases and $s_0(k,t)$ increases as $k$
and $t$ increase.

\begin{thm}\label{intheight1} For any $a\in\FF_p^*$ we have
$$\# \{|x|~:~x\in U(k,t,a)\}\le\Psi(k,p_{s+1})$$
where $s = s_0(k,t)$.
\end{thm}

\begin{thm}\label{intheight2} Let $a\in\FF_p^*$ and $x_0\in U(k,t,a)$
$$\# \{|x|~:~x\in U(k,t,a), \  (x,x_0)=1\}\le\Psi(k,p_{s})$$
where $s = s_0(k,t)$.
\end{thm}

If the coset is $G$ itself then we can take $x_0=1$. Thus, we have the following
estimate for the number of small elements in a subgroup.

\begin{cor}\label{subgr1} We have
$$\# \{|x|~:~x\in U(k,t,1)\}\le\Psi(k,p_{s})$$
where $s = s_0(k,t)$.
\end{cor}

\begin{remark} If we are interested in counting 
the number of $x$ such that $x\in U(k,t,a)$
and $1\le x\le k$ then sometimes it is possible to  estimate
this number slightly better than in
Theorem~\ref{intheight1} by replacing $r_0(k)$ in the definition of $s_0(k,t)$
with
$$\widetilde{r_0}(k) =\fl{\frac{\log p}{\log k}}.$$
A similar improvement can be made for Theorem~\ref{intheight2} as well.
\end{remark}

To study elements of small rational height in cosets of subgroups,
we also define
$$r_1(k)=\fl{r_0(k)/2}=\fl{\frac{\log(p/2)}{2\log k}}.$$
If $r_1(k)\ge1$ (that is, $k\le(p/2)^{1/2}$) we also define
$$s_1(k,t)=\max\left\{s~:~\binom{r_1(k)+s}{s}\le t\right\}.$$
Next, we denote for $s\in\NN$
$$\tPh(k,s)=\sum_{i=0}^{s}\binom{s}{i}\Psi(k,p_i)\Psi(k,p_{s-i}).$$
Also, we consider that $\tPh(k,0)=1$.

\begin{thm}\label{ratheight1} For any $a\in\FF_p^*$ we have
$$\# \{|x|~:~x\in V(k,t,a)\}\le\tPh(k,s+1)$$
where $s = s_1(k,t)$.
\end{thm}

\begin{thm}\label{ratheight2} We have
$$\# \{|x|~:~x\in V(k,t,1)\}\le\tPh(k,s)$$
where $s = s_1(k,t)$.
\end{thm}

Theorems~\ref{intheight1}--\ref{ratheight2} can be useful
only for very small $k$. For example, if $k\gg p^\delta$ with a fixed
$\delta>0$ then the estimates given by these theorems are trivial.
However, using ideas of their proofs we can
estimate $\#V(k,t,a)$ non-trivially for very small subgroups $G$ and
and not too small $k$. In particular, if $\delta\in(0,1/10)$
is fixed and $k\asymp p^{\delta}$ then the following theorem is nontrivial
for a certain range of $t$. Denote
$$r_2(k)=\fl{\frac{\log(p/2)}{8\log k}-\frac14}.$$

\begin{thm}\label{ratheight3} Let $s\in\NN$. For any $a\in\FF_p^*$
we have
$$\#V(k,t,a)\le
\max\(2\tPh(k^2,s-1),\binom{r+s}{s}^{-1}t\)$$
where $r= r_2(k)$.
\end{thm}

We also have:

\begin{thm}\label{ratheight4} Let $s\in\NN$. We have
$$\# V(k,t,1) \le\max\(2\tPh(k,s-1),\binom{r+s}{s}^{-1}t\)$$
where $r= r_2(k)$.
\end{thm}

We do not prove analogs of Theorem~\ref{ratheight3} and~\ref{ratheight4}
for integer heights. However, notice that some estimates for $\#U(k,t,a)$ can
be deduced using the trivial inequality $\#U(k,t,a)\le\#V(k,t,a)$.

\subsection{On solutions of systems of congruences}

A general problem is to estimate , for given $r$, $a_i$, $k_i$, $l_i$
$(i=1,\ldots,r)$, the number of solutions of a system of
congruences
\begin{equation}\label{congruences}
a_i x^{k_i}\equiv l_i+y_i \pmod p
\end{equation}
in the box
\begin{equation}\label{interval}
(x,y_1,\ldots,y_r) \in\Fp^*\times \prod_{i=1}^r [1,N_i].
\end{equation}

For integers  $1\le k_1<\ldots<k_r < p-1$ which satisfy
the conditions
\begin{equation}
\label{Ia}
(k_i,p-1)<p^{1-\eps},\qquad 1\le i\le r,
\end{equation}
and
\begin{equation}
\label{IIa}
(k_i-k_j,p-1)<p^{1-\eps},\qquad 1\le j < i\le r,
\end{equation}
J.~Bourgain~\cite{Bo1} has established the following result.

\begin{lem}
\label{lem:ExpSum} Given $r\in\NN$ and $\eps>0$, there is
$\delta>0$ depending only on $r$ and $\eps$, such that for a
sufficiently large prime $p$  and $1\le k_1<\ldots<k_r < p-1$
satisfying~\eqref{Ia} and~\eqref{IIa}, for
$(a_1,\ldots,a_r)\in\Fp^r\setminus\{\mathbf 0\}$ the bound holds
$$\max_{(a_1,\ldots,a_r)\in\Fp^r\setminus\{\mathbf 0\}}
\left|\sum_{x\in\mathbb F_p}
\exp\(\frac{2 \pi i}{p}\(a_1x^{k_1}+\ldots+a_r x^{k_r}\)\)\right|
< p^{1-\delta}.
$$
\end{lem}

\begin{remark}\label{best Mordell}\rm
The condition~\eqref{IIa} is essential,
as for instance the example $x - x^{(p+1)/2}$ shows.
\end{remark}
Using standard arguments, one can deduce from Lemma~\ref{lem:ExpSum}
the following.

\begin{cor}\label{asymcor}
Given $r\in\NN$ and $\eps>0$, there are $\delta>0$ and $C$,
depending only
on $r$ and $\eps$, with the following property.  If $p>C$ is a prime
and $1\le k_1<\ldots<k_r < p-1$ satisfy~\eqref{Ia} and~\eqref{IIa}
then for $(a_1,\ldots,a_r)\in\Fp^r\setminus\{\mathbf 0\}$,
$l_1,\ldots, l_r\in\Fp$, and $N_1,\ldots, N_r\in\NN$, $N_1,\ldots,N_r\le p$,
the number $N$ of solutions of the system of congruences~\eqref{congruences}
satisfies the inequalities
$$|N - N_1\cdots N_r/p^{r-1}|<p^{1-\delta}.$$
\end{cor}
In particular, we have nontrivial solutions if
$$N_1\ldots N_r>p^{r-\delta}.$$

In~\cite[Theorem~17]{BBK} the existence of solutions is proved under
weaker restrictions on differences $k_i-k_j$, namely
$$
(k_i-k_j,p-1)<\frac p B\qquad (1\le j < i\le r)
$$
instead of~\eqref{IIa}, where $B$ depends only on $r$ and $\eps$,
 which however are not enough for getting an upper estimate on the
number of solutions of the same order $N_1\cdots N_r/p^{r-1}$.

Furthermore, the estimates for the number of solutions of
two congruences are given in~\cite[Theorem~19]{BBK} in a more precise form
under the conditions
\begin{equation}
\label{eq:Cond k}
(k_i,p-1)<p^{1-\eps}\quad (i=1,2) \mand (k_1-k_2,p-1)<\frac{p-1}{2}.
\end{equation}
More precisely, for
\begin{equation}a_1,a_2\in\mathbb F_p^*,\quad
l_1,l_2\in\Fp,\quad N_1,N_2\in\NN,
\end{equation}
we define
$$
I= \left\{\begin{matrix}
x\in\mathbb F_p^*~:~ &\kern -40pt\exists (n_1,n_2)\in[1,N_1]\times[1,N_2],\\
& a_jx^{k_j}\equiv l_j+n_j \pmod p\, (j=1,2) \end{matrix}\right\}.
$$
Then by~\cite[Theorem~19]{BBK},
for every $\eps>0$ there is $\eta>0$, such that the following holds.
If $k_1$, $k_2$ satisfy~\eqref{eq:Cond k}
then for $1\le N_1,N_2\le p$ and any $\delta \in (0, \eta)$
we have
\begin{equation}\label{kon1}
\# I\geq\(\frac{N_1N_2}p-Cp^{1-\delta}\)
\(1-\max\(\frac{2(k_1-k_2,p-1)}{p-1},5\delta\)\),
\end{equation}
where $C>0$ is an absolute constant.

\begin{remark} The estimate is nontrivial if $\delta<1/5$ and
$N_1N_2 > C p^{2-\delta}$.
\end{remark}

Thus,  the bound~\eqref{kon1}
gives a nontrivial estimate for the number of
solutions under very weak assumptions on $k_1-k_2$. If
$(k_1-k_2,p-1)$ is essentially smaller than $p$ then we can get
a better lower estimate.

\begin{thm}\label{thm3}
There exists an absolute constant $C>0$,
and for every $\eps>0$ there is $\eta>0$, such that the following holds.
If $\Delta\in\NN$, $k_1$, $k_2$,  $\eta$, satisfy~\eqref{eq:Cond k}
as well as
$$
1\le N_1,N_2\le p,\quad \Delta\le p^{\eta},
$$
then
\begin{equation}
\# I\geq\(\frac{N_1N_2}p-Cp\Delta^{-1}\)
\(1-t^{-1}\max_{a\in\Fp^*}\#V(\Delta,t,a)\),
\label{kon2}
\end{equation}
where
$$
t=\frac{p-1}{(k_1-k_2,p-1)},
$$
\end{thm}

Using Theorem~\ref{thm3} and Theorem~\ref{ratheight3} one can estimate
the number of solutions of a system of two congruences from below.
We   give some related examples.

\subsection{Fixed points of the discrete logarithms}

We consider some exponential congruences which are related to studying
fixed points of the discrete logarithms in finite fields,
see~\cite{CP, CZ,Ho, HM1, HM2, Zh}.
For a prime $p$ we denote by $F(p)$
the number of solutions of the congruence
\be\label{defcongr}
g^h\equiv h \pmod p,\quad 1\le g,h\le p-1,
\ee
with arbitrary integers $g$ and $h$.
J.~Holden and P.~Moree~\cite{HM2}
have conjectured that 
\be\label{eq:HM Conj}
F(p)=(1+o(1))p.
\ee
It has been
shown in~\cite{BKS} that $F(p)=p+O(p^{4/5+\eps})$ for a set of primes of
relative density $1$.

It is noted in~\cite[Bound~(33)]{BKS}  that
$$
F (p) \le (p-1) \tau(p-1),
$$
where $\tau(n)$ is the number of divisors of $n\in\NN$.
Therefore
\begin{equation}
\label{eq:TrivBound}
F(p)\le p\exp\(\frac{(\log 2+o(1))\log p}{\log\log p}\).
\end{equation}
We note that in~\eqref{eq:TrivBound} and everywhere else
in the paper, the argument of iterated logarithms is
always assumed to be large enough so that the function
is well-defined.

Towards the conjecture \eqref{eq:HM Conj},
here we prove the following upper estimate for $F(p)$.

\begin{thm}\label{Fupper}
We have
$$
F(p) =O(p).$$
\end{thm}

We remark that though obtaining an asymptotic formula  for $F(p)$
seems to be difficult, rather elementary arguments imply the lower bound
\begin{equation}
\label{eq:LowBound}
F(p)\ge p + O(p^{3/4+o(1)}).
\end{equation}
It is likely that the arguments of~\cite{BKS} can be used to get 
a better remainder term in \eqref{eq:LowBound}; we do not try 
to optimise it in this paper.

\subsection{Notation}
We recall that $U = O(V)$, $U \ll V$ and  $V \gg U$  are all equivalent to the
statement that the inequality $|U| \le c\,V$ holds with some
constant $c> 0$.
Sometimes we write $U = O_\lambda(V)$, $U \ll_\lambda V$ and  $V \gg_\lambda U$
to emphasise that the implied constant may depend on a certain parameter
$\lambda$. 
 We also write $U\asymp V$ if $U \ll V \ll U$.

%
\section{Distribution of Elements of Cosets of Small Subgroups}
\label{main}
%

\subsection{Proof of Theorem~\ref{intheight1}}
Take a maximal possible set
$$\{x_0,x_1,\ldots,x_{\ell}\}\subseteq\ZZ$$
of multiplicatively independent elements of $U(k,t,a)$.
We claim that
\be\label{ests}
\ell \le s.
\ee
Indeed, let
$$Y=\left\{x_0^{u_0}x_1^{u_1}\ldots x_\ell^{u_\ell}\in\ZZ~:~
u_0,\ldots,u_{\ell}\ge0,\ u_0+\ldots+u_{\ell}=r\right\}$$
where $r = r_0(k)$.

For any $y\in Y$ the exponents $u_0,\ldots,u_\ell$
are uniquely defined due to multiplicative independence of
$x_0,\ldots,x_{\ell}$. Therefore,
\be\label{est|Y|}
\# Y=\binom{r+\ell}{\ell}.
\ee
Next, for any $y\in Y$ we have $|y|\le k^{r}<p/2$.
Thus, different elements of $y\in\ZZ$ are different as elements
of $\Fp$ as well. Finally, $Y\subseteq a^{\ell+1}G$. Hence, $\# Y\le t$.
Comparing this inequality with~\eqref{est|Y|} and recalling the
definition~\eqref{eq:s0} of $s_0(k,t)$ we get~\eqref{ests}.

Take the largest possible $n$ so that $p_n\le k$
and define the matrix
$$A=\(\alpha_{i,j}\)_{0\le i\le \ell}^{1\le j\le n}$$
of exponents in the prime number factoizations
$$|x_i|=\prod_{j=1}^n p_j^{\alpha_{i,j}},\quad0\le i\le \ell.$$
By the choice of $x_0,\ldots,x_s$, the matrix $A$ is of full rank $\ell+1$
and for any element
$$x=\pm \prod_{j=1}^np_j^{\alpha_j}\in U(k,t,a)$$
the vector
$(\alpha_1,\ldots,\alpha_n)$ is a linear combination of rows of the matrix
$A$ with rational coefficients. We take a non-singular
$(\ell+1)\times(\ell+1)$
submatrix $B$ of $A$. Let the columns of $B$ correspond to prime numbers
$q_1<\ldots<q_{\ell+1}$.

For any $x\in U(k,t,a)$ we define $\bb(x)=(\beta_1,\ldots,\beta_{\ell+1})$ by
$$
q_j^{\beta_{j}} \mid x \mand q_j^{\beta_{j}+1} \nmid x,
\qquad 1\le j\le \ell+1.
$$

We denote the rows of $B$
by
$$\bb_i=(\beta_{i,1},\ldots,\beta_{i,\ell+1}), \qquad 0\le i\le \ell.$$
 So,
$$B=(\beta_{i,j})_{0\le i\le \ell}^{1\le j\le \ell+1}.$$
We see that $\bb_i = \bb(x_i)$.

We claim that different elements $|x|$ with $x\in U(k,t,a)$
define different vectors $\bb(x)$. Indeed, the vector $\bb(x)$
determines rational numbers $u_0,\ldots,u_\ell$ so that
$$\bb(x)=\sum_{i=0}^{\ell}u_i\bb_i.$$
Then
$$
|x| = \prod_{i=0}^{\ell} |x_i|^{u_i}
$$
is also uniquely determined by $\bb(x)$.

It suffices to estimate the number of possible vectors $\bb(x)$
with $x\in U(k,t,a)$. We note that
$$
\prod_{j=1}^{s+1} p_j^{\beta_j(x)}
\le  \prod_{j=1}^{s+1} q_j^{\beta_j(x)} \le |x|\le k$$
We now see that  different elements $|x|$ with $x\in U(k,t,a)$ define
different
$p_{\ell+1}$-smooth numbers, thus
$$\#U(k,t,a)\le\Psi(k,p_{\ell+1}).$$
Using~\eqref{ests} completes the proof.

\subsection{Proof of Theorem~\ref{intheight2}}
Now we take a maximal possible set
$$\{x_1,\ldots,x_\ell\} \subseteq\ZZ$$
of multiplicatively independent elements of $U(k,t,a)$
with $(x_0,x_j)=1$ for $j=1,\ldots,\ell$.
The inequality~\eqref{ests} can be proved in a similar way.
If $y=x_0^{u_0}x_1^{u_1}\ldots x_s^{u_s}$ then the number $|x_0|^{u_0}$
is determined as the largest power of $|x_0|$ dividing $y$. The exponents
$u_1,\ldots,u_\ell$ are uniquely defined due to multiplicative independence of
$x_1,\ldots,x_\ell$ and  the equality
$$
x_1^{u_1}\ldots x_\ell^{u_\ell}=\pm yx_0^{-u_0}.
$$
The matrix
$$A=(\alpha_{i,j})_{1\le i\le \ell}^{1\le j\le n}$$
is defined by the equalities
$$|x_i|=\prod_j p_j^{\alpha_{i,j}},\quad1\le i\le \ell.$$
The rest of the proof is the same as in Theorem~\ref{intheight1}.

\subsection{Proof of Theorem~\ref{ratheight1}}
Similarly to the proof of Theorem~\ref{intheight1}
we take a maximal possible set
$$
\{x_0,x_1,\ldots,x_\ell\}\in\QQ
$$
of multiplicatively independent elements of $V(k,t,a)$. Now we have to
verify that
\be\label{estsrat}
\ell \le s = s_1(k,t).
\ee
To do so, we define
$$Y=\left\{x_0^{u_0}x_1^{u_1}\ldots x_\ell^{u_\ell}\in\QQ~:~
u_0,\ldots,u_\ell\ge0, \ u_0+\ldots+u_s=r\right\}$$
where $r = r_1(k)$.
The proof of~\eqref{estsrat} follows the proof of~\eqref{ests}.
The distinction is that now for any $y\in Y$ we have
$\|y\|\le k^{r}<\sqrt{p/2}$,
and thus different rational values of  $y$
with this condition get reduced to different elements of $\Fp$.

 We also choose prime numbers
$q_1<\ldots<q_{\ell+1}$ as in the proof of Theorem~\ref{intheight1}.
Now we have to estimate the number of rational numbers of the form
$$m=\prod_{j=1}^{s+1} q_j^{\beta_j}\in\QQ~:~\|m\|\le k.$$
Denote
$$J_+=\{j~:~\beta_j\ge0\},\quad m_+=\prod_{j\in J_+}q_j^{\beta_j}\in\NN,$$
$$J_-=\{j~:~\beta_j<0\},\quad m_-=\prod_{j\in J_+}q_j^{-\beta_j}\in\NN.$$
So, $m=m_+/m_-$, $m_+\le k$, $m_-\le k$. For a fixed $i=\# J_+$ there are
$\binom{\ell+1}{i}$ ways to select $J_+\subseteq{1,\ldots,\ell+1}$. As
$J_+$ and $J_-$ have been chosen, there are at most $\Psi(k,p_i)$
possible values for $m_+$ and at most $\Psi(k,p_{\ell-i+1})$
possible values for $m_-$. Combining these estimates we complete the proof.

\subsection{Proof of Theorem~\ref{ratheight2}}
We follow the proof of
Theorem~\ref{ratheight1}. Now we take a maximal possible set
$\{x_1,\ldots,x_s\}$ of multiplicatively independent elements of $V(k,t,1)$
and define
$$Y=\left\{x_1^{u_1}\ldots x_\ell^{u_\ell}\in\QQ~:~
u_1,\ldots,u_\ell\ge0,\ u_1+\ldots+u_\ell\le r\right\}.$$

\subsection{Proof of Theorems~\ref{ratheight3}
and~\ref{ratheight4}}

We   prove Theorems~\ref{ratheight3} and~\ref{ratheight4}
simultaneously.

We take a maximal possible set $\{x_1,\ldots,x_\ell\}$
of multiplicatively independent elements of $V(k,t,1)$. We consider separately
two cases.

\subsubsection*{Case of $\ell<s$.} Assume that $V(k,t,a)\neq\emptyset$ and fix
$y_0\in V(k,t,a)$. Then for any $y\in V(k,t,a)$ we have
$y=xy_0$ for some $x\in V(k^2,t,1)$. By the arguments
from the proof of Theorem~\ref{ratheight3} the number of elements $x$
satisfying these conditions is at most
$2\tPh(k^2,\ell)\le2\tPh(k^2,s-1)$ as desired
for Theorem~\ref{ratheight3}. If $a=1$ we take $y_0=1$ and
we have $x=y\in V(k,t,1)$. This gives the bound
$\# V(k,t,1)\le2\tPh(k,s-1)$ as desired for Theorem~\ref{ratheight4}.

\subsubsection*{Case of $\ell \ge s$.}
Take the largest possible $n$ so that $p_n\le k$.
Define the matrix
$$A=(\alpha_{i,j})_{1\le i\le \ell}^{1\le j\le n}$$
by equalities
$$|x_i|=\prod_{j=1}^n p_j^{\alpha_{i,j}},\quad1\le i\le \ell.$$
By the choice of $x_1,\ldots,x_s$, the matrix $A$ has rank $\ell$
and for any element
$$
x=\pm\prod_{j=1}^n p_j^{\alpha_j}\in V(k,t,1)
$$
the vector
$\ab(x)=(\alpha_1,\ldots,\alpha_n)$ is a rational linear combination of rows
of the matrix $A$. Equivalently, the vector
$(\alpha_1,\ldots,\alpha_n)$ belongs to the linear subspace $Y$ over $\QQ$
generated by $\ab(x_j)$, $j=1,\ldots,\ell$. We now assume
that the elements
$x_1,\ldots,x_n$ are chosen so that the parallelepiped
$\curly{P}(x_1,\ldots,x_\ell)$ with edges $\ab(x_j)$,
$j=1,\ldots,\ell$ has a maximal volume.

We claim that for any integer vector $(u_1,\ldots,u_\ell)\neq0$ we have
\be\label{lowheightest}
\|x\|> k^{1/2},\quad x =\prod_{j=1}^\ell x_j^{u_j}\in\QQ.
\ee
Indeed, assume that~\eqref{lowheightest} does not hold. Without loss of
generality, we can assume that $|u_1|\neq0$. We consider a new system
of multiplicatively independent elements of $V(k,t,1)$
obtained by replacing a number $x_1$ in the system $(x_1,\ldots,x_\ell)$
by $x^2$. Then the volume of $\curly{P}(x^2,x_2,\ldots,x_\ell)$ is the volume
of $\curly P(x_1,\ldots,x_\ell)$ multiplied by $2|u_1|>1$. This does not agree
with the choice of $(x_1,\ldots,x_\ell)$. So,~\eqref{lowheightest} holds.

Now we consider the set
\begin{equation*}
\begin{split}
Y=\Biggl\{\prod_{j=1}^\ell x_j^{4u_j}y~:~u_0,&\ldots,u_\ell\ge0, \\
& u_0+\ldots+u_\ell\le r,\
y\in V(k,t,1)\Biggr\}\subseteq\QQ,
\end{split}
\end{equation*}
where $r = r_2(k)$.
Using~\eqref{lowheightest} and the inequality $\|y_1/y_2\|\le k^2$
for $y_1,y_2\in V(k,t,1)$, we deduce that different
vectors $(u_1,\ldots,u_\ell,y)$
define different elements of $y$. Therefore,
\be\label{est|Y|3}
\# Y=\binom{r_2+\ell}{\ell}\# V(k,t,1).
\ee
Next, for any $y\in Y$ we have $\|y\|\le k^{4r_0+1}<\sqrt{p/2}$. Thus,
different rational values of  $y$
with this condition get reduced to different elements of $\Fp$.
Finally, $Y\subseteq aG$. Hence, $\# Y\le t$.
Comparing this inequality with~\eqref{est|Y|3} we get
$$t\ge\binom{r_2+\ell}{ \ell}\# V(k,t,1)\ge\binom{r_2+s}{s}\# V(k,t,1).$$
This completes the proof.

%
\section{Estimates for sets of smooth numbers}
%

We need several estimates for $\Psi(x,y)$.

If $y$ is small, we use the following result,
see~\cite[Theorem~1.4]{HT}:

\begin{lem}\label{lem:smoothmain} Uniformly for $x\ge y\ge 2$, we have
$$\log\Psi(x,y)=Z\(1+O\(\frac1{\log y}+\frac1{\log\log x}
\)\),$$
where
$$Z=\frac{\log x}{\log y}\log\(1+\frac y{\log x}\)
+\frac{y}{\log y}\log\(1+\frac{\log x}y\).$$
\end{lem}

In particular, we have, see~\cite[Equation~(1.14)]{HT}:

\begin{cor}\label{smooth1} For any $\alpha>1$ we have
$$\Psi(x,(\log x)^\alpha)=x^{1-1/\alpha+o(1)}\quad(x\to\infty).$$
\end{cor}

Moreover, for very small $y$ the asymptotic formula for $\Psi(x,y)$ is known,
see~\cite[Theorem~1.5]{HT}:

\begin{lem}\label{lem:smoothmain1} Uniformly for $2\le y\le(\log x)^{1/2}$,
we have
$$\Psi(x,y)=\frac1{\pi(y)!}\prod_{p_j\le y}\frac{\log x}{\log p_j}
\(1+O\(\frac{y^2}{(\log x)\log y}\)\).$$
\end{lem}

Therefore:

\begin{cor}\label{smooth1a} For any $s\in\NN$
and $x\ge\exp(p_s^2)$ we have
$$\Psi(x,p_s)\ll(\log x)^s.$$
\end{cor}

For large $y$ we know the following estimate for $\Psi(x,y)$,
see~\cite[Corollary~1.3]{HT}:

\begin{lem}\label{lem:smoothmain2} Let $x\ge y\ge 2$ and $u=(\log x)/\log y$.
For any fixed $\eps>0$ we have
$$\Psi(x,y)=xu^{-(1+o(1))u},$$
as $y$ and $u$ tend to infinity, uniformly in the range
$y\ge(\log x)^{1+\eps}$.
\end{lem}

Moreover, in a smaller range an asymptotic formula for $\Psi(x,y)$ is known, 
see~\cite[Theorem~1.1 and Corollary~2.3]{HT}.

\begin{lem}\label{lem:smoothmain3}
Let $x\ge y\ge 2$ and $u=(\log x)/\log y$.
For any fixed $\eps>0$ we have
$$\Psi(x,y)=x\rho(u)\(1+O\(\frac{\log (1+ u)}{\log y}\)\),$$
uniformly in the range
$$
y\ge\exp\((\log\log x)^{(5/3)+\eps}\),
$$
where
$$\rho(u)=\exp(-u(\log u+\log\log (u+2))+O(1)).$$
is the Dickman function.
\end{lem}

%
\section{On solutions of systems of congruences}
%

\subsection{Discrepancy and exponential sums}

We recall the notion of {\it discrepancy\/}
which for a sequence of $H$ points
\begin{equation}
\label{eq:Seq}
\Gamma = \(\gamma_{1,x}, \ldots\,,
\gamma_{m,x} \)_{x=1}^H
\end{equation}
in the $m$-dimensional unit cube,
is defined as
$$
\Delta_\Gamma = \sup_{B \subseteq [0,1)^m}
\left|\frac{T_\Gamma(B)} {H} - |B|\right|,
$$
where $T_\Gamma(B)$ is the number of points of the sequence $\Gamma$
in  the box
$$
B = [\alpha_1, \beta_1) \times \ldots \times [\alpha_{m}, \beta_{m})
\subseteq [0,1)^m
$$
of volume $|B|$
and the supremum is taken over all such boxes.

A link between the discrepancy and exponential sums has been
established independently by Koksma~\cite{Kok}
and Sz\"usz~\cite{Sz}, see also~\cite[Theorem~1.21]{DrTi}.

\begin{lem}
\label{lem:K-S} For any
integer $L > 1$ and  sequence $\Gamma$ of $H$ points~\eqref{eq:Seq}
 the following bound holds
$$
\Delta_\Gamma  \ll_m  \frac{1}{L}
+ \frac{1}{H} \sum_{\substack{\lambda_1, \ldots \lambda_m \in \ZZ\\
 0 < |\lambda_1| + \ldots + |\lambda_m| \le L}}
 \prod_{j=1}^m \frac{1}{|\lambda_j| + 1}
\left| \sum_{x=1}^H \exp \(2 \pi i\sum_{j=1}^{m}\lambda_j\gamma_{j,x} \)
\right|.
$$
\end{lem}

\subsection{The proof of Theorem~\ref{thm3} }

Let
$$d=(k_1-k_2,p-1) \mand t=(p-1)/d.$$
If $(k_1-k_2,p-1)<p^{1-\eps/2}$ then the result follows from
Corollary~\ref{asymcor}. So we assume that
$$
d\geq p^{1-\eps/2}.
$$
Write $x=y^tz$. For any $z\in\mathbb F_p^*$ we denote
$$I_z=\{y\in\mathbb F_p^*: y^tz\in I\}.$$
Then
\be\label{identI}
\# I=\frac1{p-1}\sum_{z\in\mathbb F_p^*} \# I_z.
\ee
Fix $z$. Since $y^{k_1t}\equiv y^{k_2t} \pmod p$, the condition $x\in I$ can
be written as
\be\label{zcond}
a_jy^{k_1t}z^{k_j}\equiv l_j+n_j \pmod p,\,n_j\in[1,N_j],\,j=1,2.
\ee
Now we denote
$$\cZ=\{z\in\Fp^*~:~\|a_2z^{k_2-k_1}/a_1\|>\Delta\}.$$
To get a lower estimate for $\# I$ we take a sum over only $z\in Z$.

First we estimate $\# \cZ$. The multiset $\{z^{k_1-k_2}\}$
is the subgroup $G$ of $\Fp^*$ of order $t$, and each element
of $G$ has multiplicity $d$. Therefore,
\be\label{est|Z|}
\# \cZ=\(1-t^{-1}|V(\Delta,t,a_2/a_1)|\)(p-1).
\ee

Next, we   estimate $\# I_z$ for $z\in \cZ$.
We note that  the condition $z\in Z$ implies that
$$\lambda_1a_1+\lambda_2a_2z^{k_2-k_1}\not \equiv 0 \pmod p$$
for $0<\max(|\lambda_1|,|\lambda_2|)\le\Delta$. By Lemma~\ref{lem:ExpSum}
we have
$$\left|\sum_{y\in\Fp^*}
\exp\(\frac{2 \pi i}{p}(y^{k_1t}z^{k_1}(\lambda_1a_1+\lambda_2a_2z^{k_2-k_1})\)
\right|\le p^{1-\delta}$$
for some  $\delta>0$ that depends only on $\eps$.
Thus by Lemma~\ref{lem:K-S}, applied with $m=2$ and $L = \Delta$, we obtain
$$
\# I_z = \frac{N_1N_2}{p} + O\(p/\Delta + p^{1-\delta}
 \(\log\Delta\)^2\).
$$
We observe that for $\eta\le\delta/2$ and $\Delta\le p^\eta$ we have
$$p^{1-\delta}\(\log\Delta\)^2 \ll p/\Delta.$$
Recalling~\eqref{identI} and the bound~\eqref{est|Z|}
we complete the proof.

\subsection{Applications of Theorem~\ref{thm3}: some examples }

We consider the cases when $N_1$, $N_2$ are comparable to $p$. In our examples
we use notation from the statement of Theorem~\ref{thm3} .

\begin{cor}\label{example1} For any $A_1>0$, $A_2>0$ there exists some $B$
such that if
$$
N_1\ge A_1p,\qquad N_2\ge A_1p,\qquad t\le\frac{A_2\log p}{\log\log p}
$$
then
$$\# I\geq\frac{N_1N_2}p\(1-\frac Bt\).$$
\end{cor}

Indeed, we assume that $p$ is large enough. We take $\Delta=[(\log p)^2]$.
To estimate $\#V(\Delta,t,a)$ we apply Theorem~\ref{ratheight3} with $s=1$.

\begin{cor}\label{example2}  For any $l\in\NN$, $A_1>0$, $A_2>0$ there exists
some $B$ such that if
$$
N_1\ge A_1p,\qquad N_2\ge A_1p,\qquad t\ge\frac{A_2(\log p)^l}{\log\log p}
$$
then
$$\# I\geq\frac{N_1N_2}p\(1-
\frac{B(\log\log p)^l}{(\log p)^l}\).$$
\end{cor}
Indeed, we assume that $p$ is large enough. We take $\Delta=[(\log p)^{l+1}]$.
To estimate $\#V(\Delta,t,a)$ we apply Theorem~\ref{ratheight3} with $s=l$
observing that, due to Corollary~\ref{smooth1a}, we have
$\tPh(x,s-1)\ll_s(\log x)^{s-1}$ for $x\ge2$.

If $t$ is large one can use the following estimate.

\begin{cor}\label{example3}
For any $\eta \in (0,1/2]$,  $\eps> 0$,  $A_1>0$, $A_2>0$ there exists some
$B$ such that if $t\le p^\eta$ and
$$N_1\ge A_1p,\quad N_2\ge A_1p$$
then
$$\# I\geq\frac{N_1N_2}p\(1-B\frac{\log p+t^{(1+\eps)/l}}t\)$$
where  $l=\fl{1/(2\eta)}$.
\end{cor}

\begin{proof} Again, we assume that $p$ is large enough. We take
$\Delta=[t/2]$ and denote
\begin{equation*}
\begin{split}
\AA=\{(u/v)~:~u\in\ZZ, \ v\in\NN, \ &|u|\le\Delta, \, v\le\Delta,\\
& \exists x\in aG\ vx\equiv u \pmod p\}.
\end{split}
\end{equation*}
By the definition of $V(\Delta, t,a)$, we can associate with any element
$x\in V(\Delta, t,a)$ a rational number $u/v\in\AA$ such that
$vx\equiv u \pmod p$. Thus,
\be\label{estAA}
\#\AA \le \# V(\Delta,t,a).
\ee
Actually, the equality in~\eqref{estAA} holds since $\Delta\le\sqrt{p/2}$.

Now consider the set
$$\AA^{(l)}=\{a_1\ldots a_l~:~a_1\ldots,a_l\in\AA\}.$$
By~\cite[Corollary~3]{BKS} we have
\be\label{lowestAAl}
\# \AA^{(l)} \gg \# \AA^{l-\eps}
\ee
where the implied constant in $\gg$ depends on $l$ and $\eps$.
On the other hand, any element of $\AA^{(l)}$ has the form $u/v$,
$u\in\ZZ$, $v\in\NN$, $|u|\le\Delta^l$, $v\le\Delta^l$. Since
$\Delta^l\le\sqrt{p/2}$, distinct element of $\AA^{(l)}$ are
distinct as elements of $\Fp$ as well. Considering $\AA^{(l)}$
as a subset of $\Fp$ we see that $\AA^{(l)}\subseteq a^lG$. Therefore,
\be\label{uppestAAl}
\# \AA^{(l)}\le t.
\ee
Inequalities~\eqref{estAA}, \eqref{lowestAAl} and~\eqref{uppestAAl} imply
$$\#V(\Delta,t,a)\ll t^{(1+\eps)/l}.$$
The corollary follows from this estimate and Theorem~\ref{thm3}.
\end{proof}

%
\section{Fixed points of the discrete logarithms}
%

\subsection{Preparations}

For any divisor $d$ of $p-1$ and $k=(p-1)/d$ we define the sets
$$X^*(k,p)=\left\{x~:~1\le x\le k,\ (x,k)=1,\,
x^k\equiv(-k)^k \pmod p\right\},$$
$$X(k,p)=\left\{x~:~1\le x\le k,\ 
x^k\equiv(-k)^k \pmod p\right\},$$
Let
$$T(d,p)=\# X^*(k,p).$$

It is known that
\be\label{expd}
F(p)=\sum_{d\mid p-1}dT(d,p),
\ee
see~\cite[Equation~(5)]{HM2}.

\begin{lem}\label{newX} We have
$$\sum_{h\mid k} T((p-1)/h,p)\le \# X(k,p).$$
\end{lem}

\begin{proof}
For $h\mid k$ we define the set
$$X(k,h,p)=\frac{k}{h} X^*(h,p).$$
Clearly,
\be\label{sizeX(k,h,p)}
\# X(k,h,p) = \# X^*(h,p) =T((p-1)/h,p).
\ee
For $x\in X^*(h,p)$ we have
$$\(\frac{kx}{h}\)^k
=\(\frac{k}{h}\)^k \(x^{h}\)^{k/h}
\equiv\(\frac{k}{h}\)^k \((-h)^{h}\)^{k/h}
=(-k)^k\pmod p.
$$
Therefore, the sets $X(k,h,p)$ are subsets of $X(k,p)$. Moreover, these
sets are disjoint since for $x\in X(k,h,p)$ we have $(x,k)=h/k$.
Applying \eqref{sizeX(k,h,p)} we complete the proof.
\end{proof}

We   use the following estimate for $T(d,p)$,
see~\cite[Bound~(39)]{BKS} with $\nu=3$.

\begin{lem}\label{(39)} We have
$$T(d,p)\le\(d^{-4/3}p+(p/d)^{1/3}\)(p/d)^{o(1)}.$$
\end{lem}

Following~\cite{BKS} we denote
$$D=p\exp\(-4\frac{\log p}{\log\log p}\)$$
(we always assume  that $p$ is
large enough).
Then we have
\be\label{smalld}
\sum_{\substack{d\mid p-1,\\d\le D}}dT(d,p)= p+o(p),
\ee
as $p\to\infty$,
see~\cite[Section~5]{BKS}.
It is  convenient for us to take a sum over $k$.
The identity~\eqref{expd} can be rewritten as
\be\label{expk}
F(p)=(p-1)\sum_{k\mid p-1}\frac{T((p-1)/k,p)}k.
\ee

We now define the sums
$$
S_p(K,L) = \sum_{\substack{k\mid p-1,\\K \ge k>L}}\frac{T((p-1)/k,p)}k.
$$
and denote
\begin{align*}
K_1&=\exp\(4\frac{\log p}{\log\log p}\),\\
K_2&=\exp\(\log p)^{0.4}\),\\
K_3&=\exp\((\log\log p)^7\) 
\end{align*}
So we see from~\eqref{expk} that
\be\label{eq: F and S}
\begin{split}
F(p)=(p-1)(S_p(p-1,K_1) + &S_p(K_1,K_2) \\
& + S_p(K_2,K_3) + S_p(K_3,0)).
\end{split}
\ee

Inequality~\eqref{smalld} immediately implies
\be\label{largek}
S_p(p-1,K_1) =1+o(1),
\ee
as $p\to\infty$.
So, to prove Theorem~\ref{Fupper} we have to estimate the sums 
$S_p(K_1,K_2)$, $S_p(K_2,K_3)$ and $S_p(K_3,0)$.

Our main tools are Theorem~\ref{intheight2} and estimates for sets of smooth
numbers. Besides, for the sum $S_p(K_3,0)$ we use some arguments 
which stem from functional analysis. 

\subsection{Preliminary estimates}

We recall the definitions~\eqref{eq:r0} and~\eqref{eq:s0}
of the functions $r_0(k)$ and $s_0(k,t)$.
By Lemma~\ref{newX}, Theorem~\ref{intheight2} in this case can be rewritten 
as follows:

\begin{cor}\label{cor:estT} We have
$$\sum_{h\mid k}T((p-1)/h,p)\le\# X(k,p)\le\Psi(k,p_{s})$$
where $s = s_0(k,k)$.
\end{cor}

By the prime number theorem
we have
\be\label{PNTLeg}
\frac{p_j}{\log p_j}=j+O\(\frac j{\log j}\).
\ee

It is easy to estimate the number of small divisors of a positive integer
in terms of the $\Psi(x,y)$  function. Moreover, we need to estimate the 
number of products of a small divisor of a fixed number by a small smooth 
number. For $m\in\NN$, $m\ge2$, $y>0$, and
$z>0$ by  $\tau(m,y,z)$ we denote the number of numbers $dl\le z$ where $d$
is a divisor of $m$ and $l$ is an $y$-smooth number.
Several estimates on $\tau(m,1,z)$ have recently been given in~\cite{GrTen}.
In particular our next estimate for $y=1$ is 
essentially~\cite[Bound~(7)]{GrTen}.

\begin{lem}\label{bounddivsm} We have $\tau(m,y,z)\le\Psi(z,q)$
where $q$ is the largest prime number with
$$
\prod_{\substack{y<\ell \le q\\\ell~\mathrm{prime}}}\ell \le m.
$$
\end{lem}

\begin{proof}
Let $D$ be the set of divisors $d$ of $m$ such that all prime divisors
of $d$ are greater than $y$. Thus, we have to estimate the number of
numbers $dl\le z$ where $d\in D$ and $l$ is an $y$-smooth number.
Let $q_1<q_2<\dots<q_J$ be the prime divisors of $m$ that are greater than $y$,
and let $\tilde p_1<\tilde p_2<\dots$ be the primes greater than $y$.
We associate with any number $dl$ where $l$ is $y$-smooth, $d\in D$ and
$$
d =  \prod_{j=1}^J q_j^{\alpha_j}
$$
the  number
$$l \prod_{j=1}^J \tilde p_j^{\alpha_j} \le dl,
$$
which is clearly $q$-smooth.
Since this map is injection, the result now follows.
\end{proof}

\begin{cor}\label{bounddiv} We have $\tau(m,1,z)\le\Psi(z,q)$
where $q$ is the largest prime number with
$$
\prod_{\substack{\ell \le q\\\ell~\mathrm{prime}}}\ell \le m.
$$
\end{cor}

\subsection{Some tools from functional analysis}
\label{sec:FA}

Finally, we also need the following statement which is essentially
based on linear algebra. Let $\cX$ be a linear space of real
sequences $\vec{x} = (x_q)_{q\in \cP}$, finitely supported on the
set of primes $\cP$.

Define
$$
\cH(\vec{x}) = \sum_{q \in \cP} |x_q| \log q.
$$
Also for a set $\cQ \subseteq \cP$ we use $\pi_\cQ$ to denote the
coordinate restriction on the sequences of $\cX$,
that is for  $\vec{x} = (x_q)_{q\in \cP} \in \cX$ we have
$$
\pi_\cQ(\vec{x}) = (x_q)_{q\in \cQ}.
$$

We have the following:

\begin{lem}\label{linalg}
 Let $\cF \subseteq \cX$ be a finite whose
elements have integer coordinates. Let a positive
integers $s$, $L$, $M$ be such that $L\ge M$;
\begin{enumerate}
\item  all elements of $\cF$
generate a linear subspace $\langle \cF \rangle$ of $\cX$ of dimension
$\dim \langle \cF \rangle \le s$;
\item $\cH(\vec{x}) < L$ for  $\vec{x} \in \cF$;
\item for any subset $\cQ \subseteq \cP$ with
\begin{equation}
\label{eq:Cond Q}
\sum_{q \in \cQ} \log q < M
\end{equation}
the sequences $\pi_{\cP\setminus \cQ}$ is one-to-one on $\cF$.
\end{enumerate}
Then
$$
\# \cF < (c L/M)^s
$$
for some absolute constant $c$.
\end{lem}

\begin{proof}
We take a sufficiently large prime
$q_0$ and let $\cX_0$ be (a finitely dimensional) space of all
real sequences $\vec{x} = (x_q)_{q\in \cP_0}$
where $\cP_0=\{q\in \cP, q\le q_0\}$.  
We require that $q_0$ is so large that all elements of 
$\cF$ are supported on $\cP_0$. Thus, 
we can consider $\cF$ as a subset of $\cX_0$. Reduction the lemma
to a finite dimensional subspace is caused by the using of
the Hahn-Banach separation theorem which has a simpler form
in the finite dimensional case.

Denote
$$
\fB=\{\cB \subseteq \cP_0~:~ \# \cB \le s, \ \pi_{\cB}\
\text{is one-to-one on}\ \cF\}
$$
We see from the condition~(i) that $\fB\ne \emptyset$.
Moreover if $\cQ \subseteq \cP$ satisfies~\eqref{eq:Cond Q} 
then there exists $\cB \in \fB$ with $\cB\cap \cQ = \emptyset$.
Indeed, by~(iii), $\pi_{\cP\setminus\cQ}$ is one-to-one map from
$\cF$ to $\cG = \pi_{\cP\setminus\cQ}(\cF)$. Since by~(i) the
linear subspace $\langle \cG \rangle =
\pi_{\cP\setminus\cQ}(\langle \cF \rangle) $ of $\cX_0$ is of
dimension at most $s$, there is $\cB \subseteq \cP\setminus\cQ$
with $\# \cB \le s$ and such that $\pi_\cB$ is one-to-one on
$\langle \cG \rangle$. Hence $\pi_\cB$ is one-to-one on $\cG$ and
$\cF$.

Let $\conv \fB$ be the convex hull in $\cX_0$ of characteristic functions of
all possible sets treated as elements of $\cB \in \fB$.

We claim that there exists an element 
$\beta = (\beta_q)_{q\in \cP} \in \conv \fB$
such that
\begin{equation}
\label{eq:beta} 0 \le \beta_q < 2 \frac{s}{M} \log q, \quad q \in
\cP.
\end{equation}
Indeed, assume that it is false. Define another
convex subset of $\cX_0$:
$$
\cA=\{(x_q)_{q\in \cP, q\le q_0}~:~x_q< 2 \frac{s}{M} \log q \}.
$$ 
By our assumption, the convex sets $\cA$ and $\conv \fB$ are
disjoint; also, $\cA$ is an open set. Therefore, we can apply the
Hahn-Banach separation theorem, see~\cite{Rud},  and conclude the existence 
of a
nonzero sequence $(\mu_q)_{q\in \cP}$ such that
$$
S = \sup_{(x_q)\in \cA}\sum_{q \in \cP} \mu_qx_q 
$$
satisfies
$$\sup_{(x_q)\in \cA}\sum_{q \in \cP} \mu_qx_q \le
\inf_{(x_q)\in \conv \fB}\sum_{q \in \cP} \mu_qx_q.
$$ 
and in particular is finite. Furthermore, since $S$ is finite we 
also see 
 that $\mu_q\ge0$ for all $q$. Hence
$$S = 2 \frac{s}{M} \sum_{q \in \cP} \mu_q\log q  >0.
$$ 
Now we define a nonnegative sequence
$$ 
\(\lambda_q\) = \(\mu_q/S\)_{q \in \cP}.
$$
Then we have
\begin{equation}
\label{eq:lambda1} \sum_{q \in \cP} \lambda_q \log q \le
\frac{M}{2s}
\end{equation}
and
\begin{equation}
\label{eq:lambda2}
\min_{\cB \in \fB} \sum_{q \in \cB} \lambda_q \ge 1.
\end{equation}
We now take
$$
\cQ = \left\{q \in \cP~:~ |\lambda_q|  > \frac{1}{2s}\right\}
$$
for which by~\eqref{eq:lambda1} we have~\eqref{eq:Cond Q}.
Thus there exists $\cB \in \fB$ with $\cB\cap \cQ = \emptyset$. We have
$$
\sum_{q \in \cB} |\lambda_q| \le   \frac{1}{2s} \#\cB \le  \frac{1}{2}
$$
contradicting~\eqref{eq:lambda2}. This proves the existence
of $\beta = (\beta_q)_{q\in \cP} \in \conv \fB$ satisfying~\eqref{eq:beta}.

It follows from the condition~(ii) and the
inequality~\eqref{eq:beta} that
$$
\sum_{q \in \cP} |x_q| \beta_q  <   2 \frac{s}{M} \sum_{q \in \cP} |x_q| \log q
\le  2 \frac{s L}{M}.
$$
Therefore
$$
\sum_{\vec{x} \in \cF} \sum_{q \in \cP} |x_q| \beta_q  
<   2 \frac{s L}{M}\#\cF.
$$
Since $\beta \in \conv \fB$, by the convexity there is $\cB \in \fB$
with
$$
\sum_{\vec{x} \in \cF} \sum_{q \in \cB} |x_q|   <     2 \frac{s L}{M}\#\cF.
$$
Hence there is a set $\cF_0 \subseteq \cF$ such that
\begin{equation}
\label{eq:F0 size}
\# \cF_0 \ge \frac{1}{2}\# \cF
\end{equation}
and  for every $\vec{x} \in \cF_0$ we have
\begin{equation}
\label{eq:F0 sum}
\sum_{q \in \cP} |x_q| \le \frac{4SL}M.
\end{equation}
Since $\cB \in \fB$, the map $\pi_\cB$ is one-to-one on $\cF_0$.
Moreover, $x_q \in \ZZ$ for $\vec{x} \in \cF$. Therefore, 
from~\eqref{eq:F0 size}
and~\eqref{eq:F0 sum} we derive
$$
\# \cF \le 2 \# \cF_0 = 2 \pi_\cB(\cF_0)
< 2 \binom{\rf{4s L/M}+s}{s}
$$
and the result follows.
\end{proof}

\subsection{Estimating $S_p(K_1,K_2)$}

\begin{lem}\label{K1K2} We have
$$S_p(K_1,K_2)=o(1),
$$
as $p\to\infty$
\end{lem}

\begin{proof} By Lemma~\ref{(39)}, we have
\be\label{from39}
\begin{split}
\sum_{\substack{k\mid p-1,\\K_2<k\le K_1}}&\frac{T((p-1)/k,p)}k\\
&\le\sum_{\substack{k\mid p-1,\\K_2<k\le K_1}}k^{1/3+o(1)}p^{-1/3}
+\sum_{\substack{k\mid p-1,\\K_2<k\le K_1}}k^{-2/3+o(1)}.
\end{split}
\ee
The first sum can be estimated trivially:
\be\label{k4term1}
\begin{split}
\sum_{\substack{k\mid p-1,\\K_2<k\le K_1}}k^{1/3+o(1)}p^{-1/3} &
\le\sum_{K_2<k\le K_1}k^{1/3+o(1)}p^{-1/3}\\
& \le K_1^{(4/3)+o(1)}p^{-1/3}=o(1).
\end{split}
\ee
Next,
\begin{equation*}
\begin{split}
\sum_{\substack{k\mid p-1,\\K_2<k\le K_1}}k^{-2/3+o(1)}&
\ll\sum_{\substack{k\mid p-1,\\K_2<k\le K_1}}k^{-0.65}\\
&=\sum_{K_2<k\le K_1}k^{-0.65}(\tau(p-1,k)-\tau(p-1,k-1)).
\end{split}
\end{equation*}
By partial summation,
\be\label{k4part}
\begin{split}
\sum_{\substack{k\mid p-1,\\K_2<k\le K_1}}&k^{-2/3+o(1)}\\
&\ll\sum_{K_2<k\le K_1}k^{-1.65}\tau(p-1,k)+K_1^{-0.65}\tau(p-1,K_1).
\end{split}
\ee
Define $q$ as the largest prime $q$ so that
\be\label{eq:def q}
\prod_{\substack{\ell \le q\\\ell~\mathrm{prime}}}\ell \le p-1.
\ee
By~\eqref{PNTLeg} we have
\be\label{esty0}
q=\log p+o(\log p).
\ee
Therefore, $q\le(\log k)^{2.6}$ for $k> K_2$.
Using Lemma~\ref{bounddiv} and Corollary~\ref{smooth1} we get
$$\tau(p-1,k)\ll k^{0.62}\quad(k> K_2).$$
Plugging in this estimate to~\eqref{k4part} we conclude that
\be\label{k4term2}
\sum_{\substack{k\mid p-1,\\K_2<k\le K_1}}k^{-2/3+o(1)}=o(1).
\ee
Now the result follows from~\eqref{from39}, \eqref{k4term1}
and~\eqref{k4term2}.
\end{proof}

\subsection{Estimating $S_p(K_2,K_3)$}

\begin{lem}\label{K2K3} We have
$$S_p(K_2,K_3)=o(1),
$$
as $p\to\infty$
\end{lem}

\begin{proof} For any integer $K\in(K_3,K_2]$ we   estimate
\be
\label{definS(K)}
S(K)=\sum_{\substack{k\mid p-1,\\K/e<k\le K}}\frac{T((p-1)/k,p)}k.
\ee
By Corollaries~\ref{cor:estT} and \ref{bounddiv}, we have
\be\label{est1S(K)}
S(K)\le2\Psi(K,p_s)\Psi(K,q)/K.
\ee
where $r=r_0(K)$, $s=s_0(K,K)$ and the prime $q$ is defined
by~\eqref{eq:def q}.
To estimate $s$ we   use a simple inequality
\be
\label{simple}
\binom{r+s}{s}\ge(r/s)^s.
\ee
For $K\le K_2$ we have $r\gg(\log p)^{0.6}$.
Take $s_0=\fl{(\log p)^{0.4}}$.
Then we have, by~\eqref{simple},
$$
\binom{r+s_0}{s_0}\ge(\log p)^{{0.2+o(1)}(\log p)^{0.4}}>K.
$$
Consequently, $s\le s_0=o(r)$ and
\be
\label{eq:asymp}
\binom{r+s}{s}=(r/s)^{s+o(s)}\quad \text{and}\quad
\binom{r+s+1}{s+1}=(r/s)^{s+o(s)}
\ee
Since $r/s_0\le r_0/s\le r$ we have
$$\log(r/s)\asymp\log r\asymp\log((\log p)/\log K),$$
and thus we get the order for $s$
$$s\asymp\frac{\log K}{\log((\log p)/\log K)}$$
Thus, for $K \le K_2$ we have
$$
\log(r/s) = \log r-\log s = (1+ o(1))\(\log\log p - 2\log \log K\).
$$
Hence, from~\eqref{eq:asymp} and the inequalities
$$\binom{r+s}{s}\le K < \binom{r+s+1}{s+1}
$$
we obtain an asymptotic formula for $s$:
\be\label{asymps0}
s=\frac{\log K}{\log\log p-2\log\log K}(1+o(1))+O(1)
\quad\text{as}\quad K\le K_2.
\ee
(Note the term $O(1)$ is included to have a uniform estimate
for $2\le K\le K_2$, rather than only in the range $K_3 <K \le K_2$).

Using that $\log(1 + \alpha) \le \alpha$ for any $\alpha > 0$,
and that $s\to\infty$ for $K\ge K_3$,
we now conclude from Lemma~\ref{lem:smoothmain}
that
$$
\log\Psi(K,p_s)\le(1+o(1)) \frac{p_s}{\log p_s}
\(1+\log\(1+\frac{\log K}{p_s}\)\).
$$

Next, by~\eqref{asymps0},
$$\frac{\log K}{p_s}\le\frac{\log K}{s}
\le(1+o(1))\log\log\log p
$$
thus, by the prime number theorem,
$$\frac{p_{s}}{\log p_{s}} = s+o(s).
$$
Using~\eqref{asymps0} again, we deduce
\be\label{Psi1est}
\begin{split}
\log\Psi(K,p_s)\le(1+o(1))\frac{\log K}{\log\log p-2\log\log K}
\log\log\log p\\
\le(1+o(1))\frac{5\log K}{\log\log p}\log\log\log p.
\end{split}
\ee
Using Lemma~\ref{lem:smoothmain2} and~\eqref{esty0} we get
$$
\log\Psi(K,q)\le\log K - (1+o(1))\frac{\log K}{\log\log p}
\log\frac{\log K}{\log\log p}.
$$
Thus, for $K\ge K_3$ we have the inequality
\be\label{Psi2est}
\log\Psi(K,q)\le\log K - (1+o(1))\frac{6\log K}{\log\log p}
\log\log\log p.
\ee
Combining~\eqref{Psi1est} and~\eqref{Psi2est} gives
\begin{equation*}
\begin{split}\
\log\Psi(K,p_{s})+\log\Psi&(K,q) \le\log K
-(1+o(1))\frac{\log K}{\log\log p} \log\log\log p\\
& \le \log K -(1+o(1))\frac{\log K_3}{\log\log p} \log\log\log p\\
& \le \log K-(1+o(1))(\log\log p)^6 \log\log\log p.
\end{split}
\end{equation*}

Therefore, by~\eqref{est1S(K)}, $S(K)\ll (\log p)^{-2}$. Observing that the sum
$S_p(K_2,K_3)$ 
does not exceed the sum of $O(\log p)$ of sums $S_p(K)$ with
 $K_3<K\le K_2$, we complete the proof.
\end{proof}

\subsection{Estimating $S_p(K_3,0)$}

\begin{lem}\label{K30}
We have
$$
S_p(K_3,0) = O(1).
$$
\end{lem}

\begin{proof} 
Since $X(k,p) = \emptyset$ for $k^k < p/2$ we can always assume that $k$ is
large enough.

For $K\le K_3$ we estimate the following sums similar to ones given 
by~\eqref{definS(K)}:
\be\label{definS(K)2}
S(K)=\sum_{\substack{k\mid p-1,\\K/V<k\le K}}\frac{T((p-1)/k,p)}k,
\ee
where
\be\label{definV}
V=(\log p)^{0.01}.
\ee

As before, we put $s = s_0(K,K)$.

By~\eqref{asymps0}, we have $\log s\le\log\log K-1$. Also,
$\log r\ge\log\log p-\log\log K-2$. Therefore,
\be\begin{split}
\label{uppers0}
s & \le\frac{\log K}{\log\log p-2\log\log K} \\
& \le 
\tu\(1+ O\(\frac{\log \log \log p}{\log  \log p}\)\) +O(1),
\end{split}
\ee
where
\be
\label{eq:def tu}
\tu=\frac{\log K}{\log\log p}.
\ee

Identify an integer $x$ given by its prime number factorisation
$$
x = \prod_{q \in \cP} q^{\sigma_q}
$$
with the sequence $(\sigma_q)_{q \in \cP}$. 
Then the set $X(k,p)$ becomes a subset $\cF$ of 
$\cX$ of Section~\ref{sec:FA}. 
As can be seen from the proof of Theorem~\ref{intheight1} 
(see~\eqref{ests}) it satisfies the 
condition~(i) of Lemma~\ref{linalg}.
Clearly, it also satisfies the 
condition~(ii) of Lemma~\ref{linalg} with $L = \log K$.

We now fix some $M$ with $1 < M < \log K$ (to be chosen later). 
Let $\cK_g$ be the set of ``good'' $k \le K$ for which 
the condition~(iii) of Lemma~\ref{linalg}  also holds  
for $X(k,p)$. 

Then the bound of Lemma~\ref{linalg} yields that for $k \in \cK_g$ we have
\be\label{eq:T good k}
\# X(k,p) \le \(c \frac{\log K}{M}\)^s
\ee
for some absolute constant $c>0$. 

Let $\cK_b$ be the set of the remaining ``bad'' $k \le K$ for which 
the condition~(iii) of Lemma~\ref{linalg}  fails.

Now we define a list (a multiset) $\cL$ of integers $m\in(K/V,K]$.
For any $k\in(K/V,K]$ and for any $l\le K/k$ we write $kl$
as many as $T((p-1)/k,p)$ times. Thus, we have
$$\# \cL=\sum_{\substack{k\mid p-1,\\K/V<k\le K}}T((p-1)/k,p)
\fl{\frac{K}{k}}.$$
Hence,
\be\label{eq:lowewrL}
\sum_{\substack{k\mid p-1,\\K/V<k\le K}}\frac{T((p-1)/k,p)}k
\le2K^{-1} \# \cL.
\ee

For any $m\le K$ the number of occurrences of $m$ in $\cL$ is,
by Lemma~\ref{newX},
$$\sum_{\substack{k\mid (p-1,m),\\K/V<k\le K}}T((p-1)/k,p)
\le \# X\((p-1,m),p\).$$ 
We say that $m$ is ``good'' if $(p-1,m)$ is ``good'', and 
$m$ is ``bad'' if $(p-1,m)$ is ``bad''. We split $\cL$ into the sublists
$\cL_g$ and $\cL_b$ formed by ``good'' and ``bad'' elements, respectively. 

We have
$$\# \cL_b=\sum_{\substack{k\mid p-1,\\K/V<k\le K,\\k\in\cK_b}}\# X(k,p)
\fl{\frac{K}{k}}.$$
Therefore,
\be\label{eq:badL}
\# \cL_b \le K\Sigma_b(K), \ee
where
$$
\Sigma_b(K)  
=\sum_{\substack{k\mid p-1,\\K/V<k\le K\\k\in\cK_b}}
\frac{\# X(k,p)}k.
$$
We now  estimate $\Sigma_b(K)$.

We assume that  $k \in \cK_b$. 
Then for some set $\cQ \subseteq \cP$ satisfying~\eqref{eq:Cond Q}
and  distinct integers $x_1,x_2 \in X(k,p)$ for the 
rational number $\gamma = x_1/x_2$ we have
\be
\label{eq:gamma}
\gamma = \prod_{q \in \cQ} q^{\sigma_q}, \qquad \sigma_q \in \ZZ.
\ee

Since $x_1,x_2 \in [1,k] \subseteq [1,K]$,  we have
\be
\label{eq:sum sigma}
 \sum_{q \in \cQ} |\sigma_q| \log q \le 2 \log K.
\ee

Fix an arbitrary set $\cQ$ satisfying~\eqref{eq:Cond Q} 
and estimate the number $N$ of  rational numbers $\gamma$  
satisfying~\eqref{eq:gamma} and~\eqref{eq:sum sigma}.

For some real parameter $\tau > 0$ we partition $\cQ$ as 
$\cQ = \cQ_1 \cup \cQ_2$ where $\cQ_1 = \cQ \cap[1,e^\tau]$. 
A crude estimate gives the bound
$$
N \le N_1N_2
$$
where 
\begin{eqnarray*}
N_1 & = & \# \left\{(\sigma_q)_{q \in \cQ_1}\ : \
 \sum_{q \in \cQ_1} |\sigma_q|  \le 2 \log K\right\};\\
N_2 & = & \# \left\{(\sigma_q)_{q \in \cQ_2}\ : \
 \sum_{q \in \cQ_2} |\sigma_q|  \le \frac{2 \log K}{\tau} \right\}.
\end{eqnarray*}
Since, trivially, $\# \cQ_1 \le e^\tau$ and
by~\eqref{eq:Cond Q} we also have $\# \cQ_2 < M/\tau$, 
we derive 
$$
N \le (4\log K+1)^{e^\tau}\(4  \frac{\log K}{\tau} + 1\)^{M/\tau}.
$$
Taking $\tau = 0.6 \log M$ we obtain 
$$
N  < (\log K)^{2M/\log M},
$$
provided that $K$ is large enough.

Clearly, there are at most $e^M$ possible sets $\cQ$ 
satsifying~\eqref{eq:Cond Q}. Therefore, we see that there is a finite set 
$\cU \subseteq \QQ$ (independent of $k \le K$) of cardinality 
$$
\# \cU \le e^M (\log K)^{2M/\log M}
$$ 
such that 
 if  $k \in \cK_b$ then there are 
two distinct integers $x_1,x_2 \in X(k,p)$ with $x_1/x_2 \in \cU$.

Let $r$ be the multiplicative order of $x_1/x_2$ modulo $p$. Since 
$p \mid x_1^r - x_2^r$ and $1 \le x_1,x_2 \le k$ we derive $p \le k^r$.
Hence for $k \le K \le K_3$ we have
$$
r \ge \frac{\log p}{\log k} > (\log p)^{1/2}.
$$

Let $\cR$ be the subset of all multiplicative orders modulo $p$ 
that are greater than $ (\log p)^{1/2}$  of all rational number 
$\gamma \in \cU$.  Clearly
\be
\label{eq:Bound R}
\# \cR \le \# \cU  \le e^M (\log K)^{2M/\log M}.
\ee

From the definition of $X(k,p)$ we see that $x_1^k \equiv x_2^k \pmod p$.
So if $r$ is the multiplicative order of $x_1/x_2$ then $r \mid k$.

Thus if $k \in \cK_b$ then $k \equiv 0 \pmod r$ 
for some $r \in \cR$.  
Thus, the contribution $\Sigma_b(K)$ to $S(K)$ from ``bad'' $k$ is
$$\Sigma_b(K)
=\sum_{\substack{k\mid p-1,\\K/V<k\le K\\k \in \cK_b}}\frac{\# X(k,p)}k
\le \sum_{r \in \cR} \sum_{\substack{k\mid p-1,\\K/V<k\le K\\r \mid k}}
\frac{\# X(k,p)}k .
$$
Applying Corollary~\ref{cor:estT} and then Corollary~\ref{bounddiv},
we derive 
\be\label{eq:bad}
\begin{split} 
\Sigma_b(K) & \le \frac{V\Psi(K,p_{s})}{K}  
\sum_{\substack{r \in \cR \\ r\mid p-1}} \tau((p-1)/r, K/r)\\
& \le   \frac{V\Psi(K,p_{s})}{K}  
\sum_{\substack{r \in \cR \\ r\mid p-1}} \Psi(K/r, q)
\end{split}
\ee
where, as before, the prime $q$ is defined by~\eqref{eq:def q}.

As we have noticed, only the values of $k$ with
$k^k \ge  p/2$ are of interest. So we can always assume that 
$$
K \ge \frac{\log p}{2 \log \log p}.
$$
In particular for the parameter $\tu$ given by~\eqref{eq:def tu} 
we have $\tu \gg 1$.

We also see from~\eqref{uppers0} that 
$$
\frac{p_s}{\log p_s} \le \tu \(1 + O\(\frac{1}{\log(\tu+1)}\)\) .
$$

So,  to estimate $\Psi(K,p_s)$ we use Lemma~\ref{lem:smoothmain}, where
we see that 
the corresponding values of $Z$ satisfies the inequality
$$
Z \le \frac{p_s}{\log p_s} \(1  
+ \log\(1 + \frac{\log K}{p_s}\)\).
$$
We also note that 
$$
\frac{\log K}{p_s} \gg  \frac{\log K}{\tu \log(\tu+1)}
= \frac{\log \log p}{\log(\tu+1)} \gg \frac{\log \log p}{\log \log \log p}.
$$
Therefore
\begin{equation*}
\begin{split}
Z & \le  \frac{p_s}{\log p_s} \(1  
+ \log\frac{\log K}{p_s} + o(1)\)\\
&= \tu \(1 + O\(\frac{1}{\log(\tu+1)}\)\)
\(\log\frac{\log K}{\tu \log(\tu+1)} + O(1)\)\\
&= \tu \(1 + O\(\frac{1}{\log(\tu+1)}\)\)
\(\log\frac{\log \log p}{ \log(\tu+1)} + O(1)\).
\end{split}
\end{equation*}
We now see from  Lemma~\ref{lem:smoothmain} that
$$
\log \Psi(K,p_s) \le \tu \(1 + O\(\frac{1}{\log(\tu+1)}\)\)
\(\log\frac{\log \log p}{ \log(\tu+1)} + O(1)\). 
$$
from which we derive
\be
\label{eq:PsiKps}
\log \Psi(K,p_s)  
\le \tu  
\(\log\frac{\log \log p}{ \log(\tu+1)} + 
O\(\frac{\log \log \log p}{ \log(\tu+1)}\) \). 
\ee

To estimate $ \Psi(K/r, q)$ for $r \ge r_0$ 
where  $r_0 = (\log p)^{1/2}$ we write
$$
\Psi(K/r, q) \le  \Psi(K/r_0, q).
$$
Then, using~\eqref{esty0}, for 
$$
u=\frac{\log (K/r_0)}{\log q},
$$
we obtain
$$
u = \tu - 1/2 + o(1). 
$$
Now Lemma~\ref{lem:smoothmain3} yields the estimate
\be
\label{eq:PsiKrq}
\Psi(K/r, q) \le  \Psi(K/r_0, q) \ll \frac{K}{r_0} \rho( \tu - 1/2 + o(1)) 
\ll  \frac{K}{r_0}  \tu^{-\tu}.
\ee

Substituting~\eqref{eq:PsiKps} and~\eqref{eq:PsiKrq} in~\eqref{eq:bad}, 
we derive
$$
\Sigma_b(K) \le \frac{V\#\cR}{r_0} \exp(\xi)
$$
where
\begin{eqnarray*}
\xi& = & \tu\(\log\frac{\log \log p}{ \log(\tu+1)} - \log(\tu+1) 
+ O\(\frac{\log \log \log p}{ \log(\tu+1)}\) \)\ \\
& = & \tu\(\log\frac{\log \log p}{ \tu \log(\tu+1)}
+ O\(\frac{\log \log \log p}{ \log(\tu+1)}\) \).
\end{eqnarray*}
Considering the cases $\tu \le (\log \log p)^{1/3}$
and 
$\tu>  (\log \log p)^{1/3}$ separately, we see that 
$$
\tu\frac{\log \log \log p}{ \log(\tu+1)} = 
 O\((\log  \log p)^{1/2} + \tu\) .
$$
Thus 
\be 
\begin{split}
\label{eq: Sb prelim}
\Sigma_b(K) \le V\#\cR &(\log p)^{-1/2+ o(1)}\\
&\exp\(\tu\(\log\frac{\log \log p}{ \tu \log(\tu+1)} + C\)\)
\end{split}
\ee
for some absolute constant $C> 1$. 
Considering the cases $\tu \le U_1$, $U_1 < \tu \le U_2$ and 
$\tu > U_2$ separately,
where 
$$
U_1 =  \frac{\log \log p}{(\log \log \log p)^2}
\mand 
U_2 =   e^{2C} \frac{\log \log p}{ \log \log \log p}, 
$$
we see that 
$$
\tu\(\log\frac{\log \log p}{\tu \log(\tu+1)}
+ C\) \ll \frac{\log \log p \log \log \log \log p}{ \log \log \log p}
= o(\log \log p).
$$
So inserting this bound in~\eqref{eq: Sb prelim} and 
recalling~\eqref{definV} and \eqref{eq:Bound R} we arrive to the estimate 
$$
\Sigma_b(K) \le e^M (\log K)^{2M/\log M} (\log p)^{-0.49+ o(1)}  
$$
Taking 
\be 
\label{eq: M choice}
M = 10^{-2} \log \log p
\ee
and recallinq \eqref{eq:badL} we finally derive 
\be 
\label{eq: cLb}
\# \cL_b \ll  K(\log p)^{-1/3}  .
\ee

To estimate $\# \cL_g$ we observe that for any ``good'' $m$ the number of 
occurencies of $m$ in $\cL$ is estimated by the bound~\eqref{eq:T good k} 
which with $M$ given by~\eqref{eq: M choice} becomes
$$
\# X\((p-1,m),p\) \le (c_0 \tu)^s
$$
where $c_0 = 100 c$.

The multiset $\cL$ contains only elements from the set
$$\fL=\{dl\le K~:~d\mid p-1,\ l\le V\}.$$
Hence,
\be\label{eq: estcLg1}
\# \cL_g\le(c_0 \tu)^s \#\fL.
\ee

We can estimate $ \#\fL$ by Lemma~\ref{bounddivsm} as
\be\label{eq: estL}
 \#\fL \le\Psi(K,q*),
\ee
where $q*$ is the largest prime number with
$$
\prod_{\substack{V<\ell \le q*\\ \ell~\mathrm{prime}}}\ell \le p-1.
$$
By the prime number theorem,
$$
\prod_{\substack{\ell \le q*\\ \ell~\mathrm{prime}}}\ell \le pe^{O(V)}.
$$
Using the pirme number theorem again and \eqref{definV} we get
$$q*\le\(1+o(1)\) \log\(pe^{O(V)}\)=\(1+o(1)\)\log p.$$
By Lemma~\ref{lem:smoothmain3} we have
\begin{eqnarray*}
\Psi(K, q*) & \ll & K \rho(w)\\
& = & K \exp\(-w (\log (w+1) + \log \log (w+2) + O(1))\),
\end{eqnarray*}
where 
$$
w = \tu=\frac{\log K}{\log q*} = \tu(1+o(1/\log\log p))
= \tu(1+o(\tu^{-1/6})).
$$
Therefore,
$$
\Psi(K, q*) \ll K \rho(w) = K \exp\(-\tu (\log (\tu+1) + 
\log \log (\tu+2) + O(1))\).
$$
Combining this inequality with \eqref{eq: estcLg1} and \eqref{eq: estL}
we get
\be 
\label{eq: cLg}
\# \cL_g \ll K\exp\(-\frac{1}{2} \tu \log \log (\tu+2)\).
\ee

We now see from~\eqref{eq:lowewrL}, \eqref{eq: cLb}, and~\eqref{eq: cLg} that
$$\sum_{\substack{k\mid p-1,\\K/V<k\le K}}\frac{T((p-1)/k,p)}k
\ll(\log p)^{-1/3} + \exp\(-\frac{1}{2} \tu \log \log (\tu+2)\).$$
Taking the sum over $K=(\log p)^{\nu/100}$, $100\le\nu\le100(\log\log p)^6$
with $\tu=\nu/100$ and for $K=K_3$ with $\tu=(\log\log p)^6$ we conclude 
the proof. 
\end{proof}

\subsection{Proof of Theorem~\ref{Fupper}}

Theorem~\ref{Fupper} follows immediately from the equation~\eqref{eq: F and S}, 
the asymptotic
formula~\eqref{largek} and Lemmas~\ref{K1K2}, \ref{K2K3} and~\ref{K30}.

\subsection{Lower bound}

To prove~\eqref{eq:LowBound} we recall that for
any $d \mid p-1$ we have
$$
T(d,p) = \frac{1}{d} \varphi\(\frac{p-1}{d}\) + O\(p^{1/2 + o(1)}\),
$$
where $\varphi(k)$ is the Euler function,
see~\cite[Proposition~4.3(a)]{HM2}.
Thus for any $D$, we derive from~\eqref{expd} that
\be
\label{eq:trunc}
F(p) \ge \sum_{\substack{d\mid p-1\\d \le D}}dT(d,p) =
\sum_{\substack{d\mid p-1\\d \le D}} \varphi\(\frac{p-1}{d}\)
+ O\(Dp^{1/2 + o(1)}\).
\ee
Using the trivial bound $\varphi(k)\le k$ we now obtain
\be
\label{eq:approx}
\sum_{\substack{d\mid p-1\\d \le D}} \varphi\(\frac{p-1}{d}\)
= \sum_{d\mid p-1} \varphi\(\frac{p-1}{d}\) +  O\(p^{1 + o(1)}D^{-1}\).
\ee
Also, it is known that
\be
\label{eq:sumphi}
\sum_{d\mid p-1} \varphi\(\frac{p-1}{d}\) = p-1.
\ee
Taking $D = p^{1/4}$, we see that~\eqref{eq:trunc},
\eqref{eq:approx}
and~\eqref{eq:sumphi} imply~\eqref{eq:LowBound}.

\subsection{Estimates for almost all primes}

Take an arbitrary increasing function $g:\,\RR_+\to\RR_+$ such that 
$g(u)\to\infty$ as $u\to\infty$. It is easy to see from the proof of 
Lemma~\ref{K30} that 
$$S(K_3,K_4) = o(1)
$$
where 
$$
K_4= (\log p)^{g(p)/3}.
$$
Combining this with Lemmas~\ref{K1K2} and~\ref{K2K3}, we get
$$
\sum_{\substack{k\mid p-1,\\k> K_4}}\frac{T((p-1)/k,p)}k =1+o(1). 
$$
Therefore, taking 
$$
\widetilde{K}_4 = (\log x)^{g(x)/3}$$  
for a sufficiently large $x$, 
by the arguments from~\cite[Section~5]{BBK} we conclude
that the conjecture~\eqref{eq:HM Conj}
of J.~Holden and P.~Moree~\cite{HM2} 
holds for all but at most
$$
E(x) \le \sum_{k \le \widetilde{K}_4}  \sum_{j \le k} 
\sum_{p \mid k^k - (-j)^k}  1 \ll \widetilde{K}_4^3 = (\log x)^{g(x)}$$  
primes $p \le x$, which substantially improves the bound
$$
E(x) \ll \exp\(12 \frac{\log x}{\log\log x}\).
$$
from~\cite[Section~5]{BKS}.

\section*{Acknowledgement}

The research was carried out while the second author was visiting the
Institute for Advanced Study; the
hospitality and excellent working conditions of this institution 
are gratefully appreciated. 

The research of S.~K. was supported in part
by Grant N.~11-01-00329 from the Russian Fund of Basic Researches
 and that of I.~S. by ARC grant  DP1092835.

%

\end{document}